\documentclass[12pt, a4paper]{amsart}
\usepackage{amsmath}
\usepackage{amssymb,amsfonts}
\usepackage{mathrsfs}
\usepackage{longtable}
\usepackage{amsthm}
\usepackage{verbatim}
\usepackage{multirow}
\usepackage{url}
\usepackage{color}
\theoremstyle{plain}
\newtheorem{thm}{Theorem}[section]
\newtheorem{lem}[thm]{Lemma}
\newtheorem{prop}[thm]{Proposition}
\newtheorem{cor}[thm]{Corollary}
\theoremstyle{definition}
\newtheorem{rem}[thm]{Remark}

\newcommand{\ie}{\mbox{i.e.}}
\newcommand{\Z}{\mathbb Z_+}
\newcommand{\N}{\mathbb N}
\newcommand{\F}{\mathcal F}

\begin{document}
\title{Transitive points via Furstenberg family}
\author[J. Li]{Jian Li}
\date{\today}
\address{Department of Mathematics, University of Science and Technology of China,
Hefei, Anhui, 230026, P.R. China}
\address{Wu Wen-Tsun Key Laboratory of Mathematics, USTC, Chinese Academy of Sciences}

\email{lijian09@mail.ustc.edu.cn}
\begin{abstract}
Let $(X,T)$ be a topological dynamical system and $\mathcal{F}$ be a Furstenberg family
(a collection of subsets of $\mathbb{Z}_+$ with hereditary upward property). A point $x\in X$ is called an
$\mathcal{F}$-transitive one if $\{n\in\mathbb{Z}_+:\, T^n x\in U\}\in\F$ for every
nonempty open subset $U$ of $X$; the system $(X,T)$ is called
$\F$-point transitive if there exists some $\mathcal{F}$-transitive point. In
this paper, we aim to classify transitive systems by $\mathcal{F}$-point
transitivity. Among other things, it is shown that $(X,T)$ is a
weakly mixing E-system (resp.\@ weakly mixing M-system, HY-system)
if and only if it is $\{\textrm{D-sets}\}$-point transitive
(resp.\@ $\{\textrm{central sets}\}$-point transitive,
$\{\textrm{weakly thick sets}\}$-point transitive).

It is shown that every weakly mixing system is $\mathcal{F}_{ip}$-point transitive,
while we construct an $\mathcal{F}_{ip}$-point transitive system which is not weakly mixing.
As applications, we show that every transitive system with dense small periodic sets is disjoint from
every totally minimal system and a system is $\Delta^*(\mathcal{F}_{wt})$-transitive if and only if it is
weakly disjoint from every P-system.
\end{abstract}
\maketitle

%%%%%%%%%%%%%%%%%%%%%%%%%%%%%%%%%%%%%%%%%%%%%%%%%%%%%%%%%%%%%
%%%%%%%%%%%%%%%%%%%%%%%%%%%%%%%%%%%%%%%%%%%%%%%%%%%%%%%%%%%%%

\section{Introduction}
Throughout this paper a {\em topological dynamical system}
(TDS for short) is a pair $(X, T)$,
where $X$ is a non-vacuous compact metric space with a metric $d$ and
$T$ is a continuous map from $X$ to itself.
A non-vacuous closed invariant subset $Y \subset X$ (\ie, $TY\subset Y$)
defines naturally a {\em subsystem} $(Y, T)$ of $(X, T)$.

\medskip
Let $\mathbb Z$, $\mathbb Z_+$ and $\mathbb N$ denote the sets of the integers,
the non-negative integers, and the positive integers, respectively.

\medskip

Let $(X,T)$ be a TDS, for two opene (standing for non-empty open) subsets $U,V$ of $X$, set
\[N(U,V)=\{n\in\mathbb Z_+: T^nU\cap V\neq\emptyset\}=
\{n\in\mathbb Z_+: U\cap T^{-n}V\neq\emptyset\}.\]
We call $N(U,V)$ {the \em hitting time set of $U$ and $V$}.
Recall that a system $(X,T)$ is called {\em topologically transitive}
(or just {\em transitive}) if for every two opene subsets $U,V$ of $X$
the hitting time set $N(U,V)$ is infinite.
A system $(X,T)$ is called {\em weakly mixing}
if the product system $(X\times X, T\times T)$ is transitive;
{\em strongly mixing} if for every two opene subsets $U,V$ of $X$,
the hitting time set $N(U,V)$ is cofinite, \ie, there exists some $N\in \mathbb N$
such that $N(U,V)\supset \{N, N+1,\ldots\}$.

\medskip
A system $(X,T)$ is called {\em minimal} if it contains no proper subsystem.
Each point belonging to some minimal subsystem of $(X,T)$ is called a {\em minimal point}.
Let $(X,T)$ be a transitive system,
$(X,T)$ is called a {\em P-system} if it has dense periodic points;
an {\em M-system} if it has dense minimal points;
an {\em E-system} if it has an invariant measure with full support;
a {\em topologically ergodic system} if for every two opene subsets $U,V$ of $X$,
the hitting time set $N(U,V)$ is syndetic (bounded gaps).

\subsection{Furstenberg families}
Before going on, let us recall some notations related to a family (for more details see \cite{A97}).
For the set of nonnegative integers $\mathbb{Z}_+$,
denote by $\mathcal{P}=\mathcal{P}(\mathbb{Z}_+)$ the collection of all subsets of $\mathbb{Z}_+$.
A subset $\mathcal{F}$ of $\mathcal{P}$ is called a {\em Furstenberg family} (or just {\em family}),
if it is hereditary upward, \ie, $F_1\subset F_2$ and $F_1\in\mathcal F$ imply $F_2\in \mathcal F$.
A family $\mathcal F$ is called {\em proper} if it is a nonempty proper subset of $\mathcal P$,
\ie, neither empty nor all of $\mathcal P$. Any nonempty collection $\mathcal A$ of
subsets of $\mathbb Z_+$ naturally generates a family
\[\mathcal F(\mathcal A)=\{F\subset \mathbb Z_+:\, F\supset A \text{ for some } A \in\mathcal A\}.\]

For a family $\mathcal F$, the {\em dual family} of $\mathcal F$, denoted by $\kappa\mathcal F$, is
\[\{F\in \mathcal P: F\cap F'\neq \emptyset, \forall F'\in \mathcal F\}.\]
Sometimes the dual family $\kappa\mathcal F$ is also denoted by $\mathcal F^*$.
Let $\mathcal F_{inf}$ be the family of all infinite subsets of $\mathbb Z_+$.
It is easy to see that its dual family $\kappa\mathcal F_{inf}$
is the family of all cofinite subsets, denoted by $\mathcal F_{cf}$.

{\bf All the families considered in this paper are assumed
to be proper and contained in $\mathcal F_{inf}$.}

Let $F$ be a subset of $\mathbb Z_+$, the {\em upper Banach density of $F$} is
$$BD^*(F)=\limsup_{|I|\to\infty}\frac{|F\cap I|}{|I|}$$
where $I$ is taken over all nonempty finite intervals of $\mathbb Z_+$
and $|\cdot|$ denote the cardinality of the set.
Denote by $\F_{pubd}$ the family of sets with
positive upper Banach density.

A subset $F$ of $\Z$ is called {\em thick} if it contains arbitrarily long runs of positive integers,
\ie, for every $n\in\N$ there exists some $a_n\in\Z$ such that $[a_n, a_n+n]\subset F$;
{\em syndetic} if there is $N\in\mathbb N$ such that $[n,n+N]\cap F\neq \emptyset$ for every $n\in\Z$;
{\em piecewise syndetic} if it is the intersection of a thick set and a syndetic set.
The families of all thick sets, syndetic sets and piecewise syndetic sets
are denoted by $\mathcal F_{t}$, $\mathcal F_{s}$ and $\mathcal F_{ps}$, respectively.
It is easy to see that $\kappa \mathcal F_s=\mathcal F_t$.

For a sequence $\{p_i\}_{i=1}^{\infty}$ in $\N$, define the finite sums of $\{p_i\}_{i=1}^{\infty}$ as
$$FS\{p_i\}_{i=1}^{\infty}=\left\{ \sum\nolimits_{i\in \alpha}p_i:\,
\alpha \textrm{ is a nonempty finite subset of }\N\,\right\}.$$
A subset $F$ of $\mathbb Z_+$ is called an {\em IP set}
if there exists a sequence $\{p_i\}_{i=1}^{\infty}$ in $\N$
such that $ FS\{p_i\}_{i=1}^{\infty}\subset F$.
We denote by $\mathcal F_{ip}$ the family of all IP sets.

For a family $\mathcal F$, the {\em block family of $\F$}, denoted by $b\F$, is the
family consisting of sets $F\subset\Z$ for which there exists some $F'\in \F$ such that for
every finite subset $W$ of $F'$ one has $m+W \subset F$ for some $m\in \Z$ (see \cite{HLY09}).
It is easy to see that
$b\F=\allowbreak\{F\subset\Z:\ \exists F'\in\F, (\forall n\in\N)(\exists a_n\in\Z)
\textrm{ s.t.\@ } a_n +(F' \cap [0, n])\subset F\}$,
$b(b\F)=b\F$, $b\F_{cf}=\F_t$, $b\F_{s}=\F_{ps}$ and $b\F_{pubd}=\F_{pubd}$.

For a subset $F$ of $\Z$, let $F-F$ be $\{a-b: a,b\in F$ and $a>b \}$ if $0\not\in F$,
or $\{a-b: a,b\in F$ and $a\geq b \}$ if $0\in F$.
For a family $\mathcal F$, define {\em the difference family of $\F$} as
\[\Delta(\F)=\{F\subset\Z: \exists F'\in\F, \textrm{ s.t. } F'-F'\subset F\}\]
and $\Delta^*(\F)=\kappa\Delta(\F)$.

\begin{lem}
Let $\F$ be a family, then $\Delta(\F)=\Delta(b\F)$.
\end{lem}
\begin{proof}
Since $\F\subset b\F$, $\Delta(\F)\subset \Delta(b\F)$.
If $F\in \Delta(b\F)$, there exists $F_1\in b\F$ such that $F_1-F_1\subset F$.
Since $F_1\in b\F$, there exists $F_2\in \F$ and $\{a_n\}_{n=1}^\infty $ in $\Z$ such that
$a_n+F_2\cap[0,n]\subset F_1$. Then $F_2-F_2\subset F_1-F_1$ and $F\in \Delta(\F$).
Hence, $\Delta(\F)=\Delta(b\F)$.
\end{proof}

\subsection{Three ways to classify transitive systems}
It is well known that the study of transitive systems and its
classification play a big role in topological dynamics.
There are several ways to classify transitive systems.
The first one was started with Furstenberg
by the hitting time sets of two opene subsets.
Let $\F$ be a family,
we call $(X,T)$ is {\em $\F$-transitive} if for every two opene subsets
$U, V$ of $X$ the hitting time set $N(U,V)\in \F$.
In his seminal paper \cite{F67}, Furstenberg showed
that a TDS $(X,T)$ is weakly mixing if and only if it is $\{\textrm{thick sets}\}$-transitive.

The second way is by weak disjointness.
Two TDSs are called {\em weakly disjoint} if their product system is transitive \cite{P72}.
Then a system is weakly mixing if and only if it is weakly disjoint from itself.

The third way is by the complexity of open covers
which was introduced in \cite{BHM00}.
For a TDS $(X, T)$ and a finite open cover $\alpha$ of $X$,
let $N(\alpha)$ denote the number of sets in a finite subcover of $\alpha$ with small cardinality.
For an infinite set $A=\{a_1<a_2<\cdots \} \subset \Z$ put
\[ c_A(\alpha,n)= N(T^{-a_1}\alpha\vee T^{-a_2}\alpha\vee \cdots\vee T^{-a_n}\alpha).\]
We call $c_A(\alpha,n)$ the {\em complexity function} of the cover $\alpha$ along $A$.
An open cover $\alpha =\{U_1,U_2,\cdots,U_k\}$ of $X$ is called {\em non-trivial}
if $U_i$ is not dense in $X$ for each $1\le i \le k$.
Let $\F$ be a family,
we call $(X,T)$ is {\em $\F$-scattering} if for every $A\in \F$ and non-trivial open
cover $\alpha$ of $X$ the complexity function  $c_A(\alpha,n)$ is unbounded.
%Especially, a system $(X, T)$ is called {\em scattering} if it is $\{\Z\}$-scattering.

Recently, the authors in \cite{A97, AG01,BHM00, G04, HY02, HY04, SY04}
have successfully classified transitive systems by the above three ways.
We summarize the results in the following table:

\begin{longtable}{|c|c|c|c|}
\hline
Transitive properties  & $\F$-transitivity & $\F$-scattering & \multicolumn{1}{c|}{Weakly disjointness}\\
\hline
\multirow{2}{*}{Transitivity} & \multirow{2}{*}{$\F_{inf}$}  & \multirow{2}{*}{No?}& Mild mixing \\
&&&or trivial system\\
\hline
Total transitivity & $\kappa\F_{rs}$ & No? & Periodic system \\
\hline
\multirow{2}{*}{Weakly scattering} & \multirow{2}{*}{?} & \multirow{2}{*}{No?} & Minimal equi- \\
&&& continuous system \\
\hline
\multirow{2}{*}{Scattering} & \multirow{2}{*}{$\Delta^*(\F_{ps})$} & \multirow{2}{*}{$\F_{ps}$ or $\{\Z\}$}
&Minimal system \\
&&&  or M-system \\
\hline
Strong scattering & $\Delta^*(\F_{pubd})$  & $\F_{pubd}$ or $\F_{pud}$ & E-system \\
\hline
\multirow{2}{*}{Extreme scattering} & \multirow{2}{*}{?} & \multirow{2}{*}{?} & Topologically\\
&&&ergodic system \\
\hline
Weak mixing & $\F_t$ & ? &  Itself \\% $\times$
\hline
Mild mixing &$\Delta^*(\F_{ip})$ & $\F_{ip}$ & Transitive system \\
\hline
Full scattering & ? & $\F_{inf}$ &  \multicolumn{1}{c|}{No}\\
\hline
Strong mixing & $\F_{cf}$ & No & \multicolumn{1}{c|}{No}\\
\hline
\end{longtable}
%\end{table}

%For undefined notions, please see Section 2 or \cite{HY04}.
We understand this table in the following way.
For example, a system is scattering if and only if it is $\Delta^*(\F_{ps})$-transitive
if and only if it is  $\F_{ps}$-scattering if and only if it is $\{\Z\}$-scattering
if and only if it is weakly disjoint from every minimal system
if and only if it is weakly disjoint from every M-system.

\subsection{A new way to classify transitive systems}
In this paper we propose a new way to classify transitive systems.
Before going on, we recall some basic notions.

\medskip
For $x\in X$, denote the orbit of $x$ by $Orb(x,T)=\{x,Tx, T^2x, \ldots\}$.
Let $\omega(x,T)$ be the $\omega$-limit set of $x$, \ie, $\omega(x,T)$ is the limit set of $Orb(x,T)$.
A point $x\in X$ is called a {\em recurrent point} if $x\in\omega(x,T)$; a {\em transitive point}
if $\omega(x,T)=X$. It is easy to see that a system $(X,T)$ is transitive if and only if
the set of all transitive points, denoted by $Trans(X,T)$, is a dense $G_\delta$ subset of $X$ and
$(X,T)$ is minimal if and only if $Trans(X,T)=X$.

\medskip

Let $(X,T)$ be a TDS, for $x\in X$ and an opene subset $U$ of $X$, set
\[N(x,U)=\{n\in\Z:\, T^nx\in U\}. \]
We call $N(x,U)$ {the \em entering time set} of $x$ into  $U$.
Then it is easy to see that a point $x\in X$ is a transitive point if and only if
for every opene subset $U$ of $X$ the entering time set $N(x,U)$ is infinite.
This suggests that we can use  the entering time set of a point into an opene subset
to characterize transitive points.
Let $\F$ be a family, a point $x\in X$ is called an {\em $\F$-transitive point}
if for every opene subset $U$ of $X$ the entering time set $N(x, U)\in\F$.
A system $(X,T)$ is called {\em $\F$-point transitive} if there exists some $\F$-transitive point.
In this paper, we aim to classify transitive systems by $\F$-point transitivity.
We summarize our results in the following table:

\begin{table}[!hbt]
\begin{tabular}{|c|c||c|c|}
\hline
 \multirow{2}{*}{Transitive properties} & $\F$-point & \multirow{2}{*}{Transitive properties}  &$\F$-point \\
&transitivity& &transitivity\\
\hline
Transitive system & $b\F_{ip}$ & Weakly mixing system & ?\\ \hline
E-system & $\F_{pubd}$ & Weakly mixing E-system & $\F_D$ \\ \hline
M-system & $\F_{ps}$ & Weakly mixing M-system & $\F_{cen}$\\ \hline
Transitive system with & \multirow{2}{*}{$b\F_{wt}$}
& \multirow{2}{*}{HY-system} & \multirow{2}{*}{$\F_{wt}$}\\
dense small periodic sets&&& \\ \hline
\end{tabular}
\end{table}

This paper is organized as follows.
In Section 2, we discuss the connection between families and topological dynamics and
show some basic properties of $\F$-point transitivity.
In Sections 3 and 4, we prove the main results which are declared in the above table.
It is shown that every weakly mixing system is $\F_{ip}$-point transitive, while
we construct an $\F_{ip}$-point transitive system which is not weakly mixing.
In the final section as applications, we discuss disjointness and weak disjointness.
We show that every transitive system with dense small periodic sets is disjoint from
every totally minimal system and a system is $\Delta^*(\F_{wt})$-transitive if and only if it is
weakly disjoint from every P-system.

\subsection*{Acknowledgement}
This work was supported in part  by
the National Natural Science Foundation of China (Nos.\@ 11001071, 11071231, 10871186).
The author wish to thank Prof.\@ Xiangdong Ye
for the careful reading and helpful suggestions.
The author also thanks the referee for his/her helpful
suggestions concerning this paper.

\section{Preliminary}

The idea of using families to
describe dynamical properties goes back at least to Gottschalk and Hedlund \cite{GH55}.
It was developed further by Furstenberg \cite{F81}.
For a systematic study and recent results, see \cite{A97}, \cite{G04}, \cite{HY04} and \cite{HY05}.

Let $(X,T)$ be a TDS and $\F$ be a family,
a point $x\in X$ is called an {\em $\F$-recurrent point}
if for every neighborhood $U$ of $x$ the entering time set $N(x, U)\in\F$.
Denote the set of all $\F$-recurrent points by $Rec_{\F}(X,T)$.

It is well known that the following lemmas hold
(see, e.g., \cite[Theorem 1.15, Theorem 1.17, Theorem 2.17]{F81})

\begin{lem} Let $(X, T)$ be a TDS and $x\in X$. Then
\begin{enumerate}
\item $x$ is a minimal point if and only if it is an $\mathcal F_{s}$-recurrent point.
\item $x$ is a recurrent point if and only if it is an $\mathcal F_{ip}$-recurrent point.
\end{enumerate}
\end{lem}

Recall that a TDS $(X,T)$ is called {\em $\mathcal F$-transitive}
if for every two opene subsets $U,V$ of $X$ the hitting time set $N(U,V) \in\mathcal F$;
{\em $\mathcal F$-mixing} if $(X\times X, T\times T)$ is $\mathcal F$-transitive.

\begin{lem}[\cite{F67,A97}] Let $(X, T)$ be a TDS and $\mathcal F$ be a family. Then
\begin{enumerate}
\item $(X, T )$ is weakly mixing if and only if it is $\mathcal F_{t}$-transitive.
\item $(X, T )$ is strongly mixing if and only if it is $\mathcal F_{cf}$-transitive.
\item $(X,T)$ is $\mathcal F$-mixing if and only if it is $\mathcal F$-transitive and weakly mixing.
\end{enumerate}
\end{lem}

Recall that a TDS $(X,T)$ is called {\em $\F$-center}
if for every opene subset $U$ of $X$ the hitting time set $N(U,U)\in\F$.

It is well known that a system $(X,T)$ is transitive if and only if there exists some transitive point.
It is interesting that how to characterize transitive systems by transitive points via a family.

Let $(X,T)$ be a TDS and $\mathcal F$ be a family,
a point $x\in X$ is called an {\em $\mathcal F$-transitive point}
if for every opene subset $U$ of $X$ the entering time set $N(x, U)\in\mathcal F$.
Denote the set of all $\mathcal F$-transitive points by $Trans_{\mathcal F}(X,T)$.
The system $(X,T)$ is called {\em $\F$-point transitive} if there exists some $\F$-transitive point.

Though the terminology ``$\F$-point transitivity'' is first  introduced in this paper,
the idea has appeared in several literatures, such as \cite{HPY07,HY05}.
We state their results in our way as

\begin{thm} \label{thm:E-M-system}
Let $(X,T)$ be a TDS. Then
\begin{enumerate}
\item $(X,T)$ is an E-system if and only if it is $\F_{pubd}$-point transitive if and only if
$Trans_{\F_{pubd}}(X,T)=Trans(X,T)\neq\emptyset$ (\cite{HPY07}).
\item $(X,T)$ is an M-system if and only if it is $\F_{ps}$-point transitive if and only if
$Trans_{\F_{ps}}(X,T)=Trans(X,T)\neq\emptyset$ (\cite{HY05}).
\end{enumerate}
\end{thm}

The following remark shows some basic facts about $\F$-point transitivity.
\begin{rem}
(1). It is easy to see that $x\in X$
is an $\mathcal F$-transitive point if and only if for every $F\in \kappa \mathcal F$
one has $\{T^nx: n\in F\}$ is dense in X.

(2). If $x\in X$ is an $\F$-transitive point, then so is $Tx$.
Thus, if $(X,T)$ is $\F$-point transitive then $Trans_{\F}(X,T)$ is dense in $X$.

(3). It should be noticed that $\F$-transitivity may differ greatly from $\F$-point transitivity.
For example, $(X,T)$ is $\F_t$-transitive if and only if it is weakly mixing,
but if $Trans_{\F_t}(X,T)\neq\emptyset$ then $X$ must be a singleton.
\end{rem}

Similarly to $\F$-center, we can define $\F$-point center.
A system $(X,T)$ is called {\em $\F$-point center}
if for every opene subset $U$ of $X$ there exists $x\in U$
such that the entering time set $N(x,U)\in\F$.

\begin{lem}\label{lem:F-point-center}
Let $(X,T)$ be a TDS and $\F$ be a family. If $(X,T)$ is transitive and $\F$-point center, then
$Trans_{b\F}(X,T)=Trans(X,T)$.
\end{lem}
\begin{proof}
Let $x$ be a transitive point and $U$ be an opene subset of $X$.
Since $(X,T)$ is $\F$-point center, there exists $y\in U$ such that $N(y,U)\in \F$.
For every finite subset $W$ of $N(y,U)$, by the continuity of $T$, there exists $m\in\mathbb Z_+$
such that $m+W\subset N(x,U)$. Then $N(x,U)\in b\F$ and $x$ is a $b\F$-transitive point.
\end{proof}

\begin{rem}
(1). Since every recurrent point is $\F_{ip}$-recurrent, every transitive system
is $\F_{ip}$-point center. Thus, a system $(X,T)$ is transitive if and only if
it is $b\F_{ip}$-point transitive.

(2). If $(X,T)$ is $\F$-point transitive, then it is $\F$-point center,
so by Lemma \ref{lem:F-point-center} $Trans_{b\F}(X,T)=Trans(X,T)$.
In particular, if $b\F=\F$, then $(X,T)$ is $\F$-point transitive if and only if
$Trans_{\F}(X,T)=Trans(X,T)\neq\emptyset$.
\end{rem}

%%%%%%%%%%%%%%%%%%%%%%%%%%%%%%%%%%%%%%
\section{Weakly mixing systems}
In this section, we consider weakly mixing systems. We should use the following useful lemma:
\begin{lem}[Ulam \cite{A04}]
Let $X$ be a compact metric space without isolated points and
$R$ be a  dense $G_\delta$ subset of $X\times X$.
Then there exists  a dense $G_\delta$ subset $Y$ of $X$
such that for every $x\in Y$, $R(x)=\{y\in X:\, (x,y)\in R\}$ is a dense $G_\delta$ subset of $X$.
\end{lem}

Recall that a sequence $F$ in $\Z$ is called a {\em Poincar\'e sequence} if for any
measure-preserving system $(X, \mathcal B, \mu, T)$ and $A \in\mathcal B$ with $\mu(A) > 0$
there exists $0\neq n\in F$ such that $\mu(A \cap T^{-n}A)>0$.
Let $\F_{Poin}$ denote the family of all Poincar\'e sequences.
It is well known that $\F_{Poin}=\Delta^*(\F_{pubd})$ \cite{F81,W00}.

\begin{thm}\label{thm:weak-mixing}
Let $(X,T)$ be a TDS. Consider the following conditions:
\begin{enumerate}
\item $(X,T)$ is weakly mixing.
\item $Trans_{\F_{ip}}(X,T)$ is residual in $X$.
\item $(X,T)$ is $\F_{ip}$-point transitive.
\item $(X,T)$ is $\F_{Poin}$-transitive.
\end{enumerate}
Then (1)$\Rightarrow$(2)$\Rightarrow$(3)$\Rightarrow$(4).
In addition, if $(X,T)$ is an E-system, then (4)$\Rightarrow$(1).
\end{thm}
\begin{proof}
(1)$\Rightarrow$(2) Let $R=Trans(X\times X, T\times T)$,
then $R$ is a dense $G_\delta$ subset of $X\times X$.
Since $(X,T)$ is transitive, $X$ has no isolated points.
By Ulam Lemma, there exists a dense $G_\delta$ subset $Y$ of $X$
such that for every $x\in Y$, $R(x)$ is a dense $G_\delta$ subset of $X$.
Then it suffices to show that $Y\subset Trans_{\F_{ip}}(X,T)$.

Let $x\in Y$ and $y\in R(x)$, we have $(y,y)\in \overline{Orb((x,y), T\times T)}$,
then by the following claim, $x$ is an $\F_{ip}$-transitive point.

\smallskip
{\bf Claim}: If $(y,y)\in \overline{Orb((x,y), T\times T)}$,
then for every neighborhood $U$ of $y$ the entering time set $N(x,U)$ is an IP set.

{\bf Proof of the Claim}:
For every neighborhood $U$ of $y$, let $U_1=U$
then there exists $p_1\in\mathbb N$ such that
$T^{p_1}x\in U_1$ and $T^{p_1}y\in U_1$. Let $U_2=U_1 \bigcap T^{-p_1}U_1$,
then $U_2$ is a neighborhood of $y$, so there exists $p_2\in\mathbb N$ such that
\[T^{p_2}x\in U_2\textrm{ and } T^{p_2}y\in U_2. \]
Then for every $m\in FS\{ p_i\}_{i=1}^2$
\[T^mx\in U \textrm{ and } T^my\in U.\]
We continue inductively. Assume $p_1,p_2,\ldots,p_n$ have been found such that
for every $m\in FS\{ p_i\}_{i=1}^n$
\[T^mx\in U \textrm{ and } T^my\in U.\]
Let $U_{n+1}=U \bigcap (\bigcap_{m\in FS\{ p_i\}_{i=1}^n} T^{-m}U)$, then $U_{n+1}$ is a
neighborhood of $y$, so there exists $p_{m+1}\in\mathbb N$ such that
\[T^{p_{m+1}}x\in U_{m+1}\textrm{ and } T^{p_{m+1}}y\in U_{m+1}. \]
Then for every $m\in FS\{ p_i\}_{i=1}^{n+1}$
\[T^mx\in U \textrm{ and } T^my\in U.\]
Thus, $FS\{p_i\}_{i=1}^\infty \subset N(x,U)$.

(2)$\Rightarrow$(3) is obvious.

(3)$\Rightarrow$(4) follows from the fact that every IP set is a Poincar\'e sequence (\cite[p74]{F81}).

In addition, if  $(X,T)$ is an E-system, we show that (4)$\Rightarrow$(1).
It is sufficient to show that for every opene subsets $U_1, U_2, V$ of $X$,
$N(U_1, U_2)\cap N(V,V)\neq\emptyset$.
Since $(X,T)$ is an E-system,  there exists an invariant measure $\mu$ with full support.
Then $(X,\mathcal B_X,\mu, T)$ is a measure dynamical system and $\mu(V)>0$.
Since $N(U_1,U_2)\in \F_{Poin}$, by the definition of Poincar\'e sequence one has
 $N(U_1,U_2)\cap N(V,V)\neq\emptyset$.
\end{proof}

\begin{cor}
Let $(X,T)$ be a minimal system. Then $(X,T)$ is weakly mixing if and only if it is $\F_{ip}$-point transitive.
\end{cor}

In \cite{HY02}, the authors constructed an $\F_{Poin}$-transitive system which is not weakly mixing,
we show that:

\begin{prop}\label{prop:Fip-not-wm}
There exists an $\F_{ip}$-point transitive system which is not weakly mixing.
\end{prop}
\begin{proof}
Since the construction is somewhat long and complicated, we leave it to the appendix.
\end{proof}

\subsection{Weakly mixing M-system} In this subsection, we characterize weakly mixing M-systems.
To this end, we need the concept of the central set which was first introduced in \cite{F81}.

Let $(X,T)$ be a TDS, a pair $(x,y)\in X\times X$ is called {\em proximal} if there exists a
sequence $\{n_i\}_{i=1}^\infty$ in $\N$ such that $\lim_{i\to\infty}T^{n_i}x=\lim_{i\to\infty}T^{n_i}y$.
A subset $F$ of $\Z$ is called a {\em central set}, if there exists a system
$(X,T)$, $x\in X$, a minimal point $y\in X$ and a neighborhood of $U$ of $y$
such that $(x,y)$ is proximal and $N(x,U)\subset F$.
Let $\F_{cen}$ denote the family of all central sets.

\begin{lem}\cite[Proposition 8.10]{F81}
$\F_{cen}\subset \F_{ip}\bigcap \F_{ps}$.
\end{lem}

\begin{thm}[Akin-Kolyada \cite{AK03}]
Let $(X,T)$ be a TDS. If $(X,T)$ is weakly mixing, then for every $x\in X$
the proximal cell $Prox(x)=\{y\in X: (x,y)\textrm{ is proximal}\}$
is a dense $G_\delta$ subset of $X$.
\end{thm}

\begin{thm}
Let $(X,T)$ be a TDS. Then the following conditions are equivalent:
\begin{enumerate}
\item $(X,T)$ is a weakly mixing M-system.
\item $Trans_{\F_{cen}}(X,T)$ is residual in $X$.
\item $(X,T)$ is $\F_{cen}$-point transitive.
\end{enumerate}
\end{thm}
\begin{proof}
(1)$\Rightarrow$(2). Let $\{y_n\}_{n=1}^\infty$ be a  sequence of minimal points which is dense in $X$.
By Akin-Kolyada Theorem, $\bigcap_{n=1}^\infty Prox(y_n)$ is a dense $G_\delta$ subset of $X$.
Then it suffices to show that $\bigcap_{n=1}^\infty Prox(y_n)\subset Trans_{\F_{cen}}(X,T) $.

Let $x\in \bigcap_{n=1}^\infty Prox(y_n)$ and $U$ be an opene subset of $X$.
There exists $n\in \mathbb N$ such that $y_n\in U$.
Since $(x,y)$ is proximal and $y$ is a minimal point,
by the definition of central set, we have $N(x, U)\in \F_{cen}$.
Thus, $x$ is a $\F_{cen}$-transitive point.

(2)$\Rightarrow$(3) is obvious.

(3)$\Rightarrow$(1).
By $\F_{cen}\subset \F_{ps}$ and Theorem \ref{thm:E-M-system}, $(X,T)$ is an M-system.
Therefore, by Theorem \ref{thm:weak-mixing} and $\F_{cen}\subset \F_{ip}$, $(X,T)$ is weakly mixing.
\end{proof}

\subsection{Weakly mixing E-system} In this subsection, we characterize weakly mixing E-systems.
To this end, we need the concept of the D-set which was first introduced in \cite{BD08}.

A subset $F$ of $\Z$ is called a {\em D-set} if there exists a system
$(X,T)$, $x\in X$, an $\F_{pubd}$-recurrent point $y\in X$ and a neighborhood of $U$ of $y$
such that $(y,y)\in \overline{Orb((x,y),T\times T)}$ and $N((x,y), U\times U)\subset F$.
Let $\F_D$ denote the family of all D-sets.

\begin{lem}[\cite{BD08}]
$\F_D\subset \F_{ip}\bigcap \F_{pubd}$.
\end{lem}
\begin{thm}
Let $(X,T)$ be a TDS. Then the following conditions are equivalent:
\begin{enumerate}
\item $(X,T)$ is a weakly mixing E-system.
\item $Trans_{\F_{D}}(X,T)$ is residual in $X$.
\item $(X,T)$ is $\F_{D}$-point transitive.
\end{enumerate}
\end{thm}
\begin{proof}
(1)$\Rightarrow$(2).
It is easy to see that  $(X\times X,T\times T)$ is also an E-system.
Let $R=Trans_{pubd}(X\allowbreak\times X, T\times T)$, then
by Theorem \ref{thm:E-M-system} $R$ is a dense $G_\delta$ subset of $X\times X$.
By Ulam Lemma, there exists a dense $G_\delta$ subset $Y$ of $X$
such that for every $x\in Y$, $R(x)$ is a dense $G_\delta$ subset of $X$.
Then it suffices to show that $Y\subset Trans_{\F_{D}}(X,T)$.

Let $x\in Y$ and $U$ be an opene subset of $X$. Choose $y\in R(x)\cap U$,
then $y$ is $\F_{pubd}$-recurrent.
Since $(x,y)$ is a transitive point in $X\times X$,
$(y,y)\in \overline{Orb((x,y),T\times T)}$.
Clearly,  $N((x,y), U\times U)\subset N(x,U)$,
so by the definition of D-set one has $N(x, U)\in \F_D$.
Thus, $x$ is an $\F_{D}$-transitive point.

(2)$\Rightarrow$(3) is obvious.

(3)$\Rightarrow$(1). By $\F_D\subset \F_{pubd}$ and Theorem \ref{thm:E-M-system}, $(X,T)$ is an E-system.
Therefore, by Theorem \ref{thm:weak-mixing} and $\F_D\subset \F_{ip}$, $(X,T)$ is weakly mixing.
\end{proof}

\begin{rem} (1). There exists a weakly mixing system which is not an E-system (\cite{HZ02,HLY09}).

(2). There exists a weakly mixing E-system which is not an M-system (\cite{BD08}).
\end{rem}

\section{Systems with dense small periodic sets}
In this section, we characterize transitive systems with dense small periodic sets.
Let $(X,T)$ be a TDS, we call $(X,T)$ has {\em dense small periodic sets} (\cite{HY05}) if
for every open subset $U$ of $X$ there exists a closed subset $Y$ of $U$ and $k \in\mathbb  N$ such
that $Y$ is invariant for $T^k$ (\ie, $T^kY\subset Y$).
Clearly, every P-system has dense small periodic sets.
If $(X,T)$ is transitive and has dense small periodic sets,
then it is an M-system.

To characterize the system with dense small periodic set,
we need a new kind of subsets of $\Z$. A subset $F$ of $\mathbb Z_+$ is called {\em weakly thick}
if there exists some $k\in\mathbb N$ such that
$\{n\in\mathbb Z_+: kn\in F\}$ is thick.
Let $\mathcal F_{wt}$ denote the family of all weakly thick sets.

\begin{lem}\label{lem:dsps}
Let $(X,T)$ be a TDS.  Then the following conditions are equivalent:
\begin{enumerate}
\item $(X,T)$ has dense small periodic sets.
\item For every opene subset $U$ of $X$, there exists $k\in\mathbb N$
such that $U$ contains a minimal subsystem of $(X,T^k)$.
\item It is $b\F_{wt}$-point center.
\end{enumerate}
\end{lem}
\begin{proof}
(1)$\Rightarrow$(2) follows from the fact that every system constants a minimal subsystem.

(2)$\Rightarrow$(3). Let $U$ be an opene subset of $X$, then
there exists $k\in\mathbb N$ such that $U$ contains a minimal subsystem $Y$ of $(X,T^k)$.
Choose a point $y\in Y\subset U$. Then $k\N\subset N(y,U)$ and $(X,T)$ is $b\F_{wt}$-point center
since $k\N\in b\F_{wt}$.

(3)$\Rightarrow$(1). Let $U$ be an opene subset of $X$.
Choose an opene subset $V$ of $X$ such that $\overline{V}\subset U$.
Since $(X,T)$ is $b\F_{wt}$-point center, there exists $x\in V$ such that $N(x, V)\in b\F_{wt}$.
Let $F=N(x,V)$, then there exist two sequences $\{a_i\}_{i=1}^\infty$, $\{n_i\}_{i=1}^\infty$ in $\Z$
and $k\in \N$ such that
$\bigcup_{i=1}^\infty(a_i+ k[n_i, n_i+i])\subset F$.
Without lose of generality, assume that
$\lim_{i\to \infty}T^{a_i+kn_i}x \allowbreak=y\in \overline{\{T^n x: n\in F\}}\subset \overline{V}$.
Let $Y=\overline{Orb(y,T^k)}$. Clearly, $T^k Y\subset Y$.
Then it suffices to show that $Y\subset U$.

For every $m\in\N$, if $i>m$, then $a_i+k(n_i+m)\in F$, so
\[T^{km}y=\lim_{i\to\infty}T^{a_i+k(n_i+m)}x\in \overline{\{T^n x: n\in F\}}\subset \overline{V}. \]
Thus, $Y=\overline{\{T^{km}y: m\in\N\}}\subset\overline{V}\subset U$.
\end{proof}

\begin{thm}\label{thm:trans-periodic-set}
Let $(X,T)$ be a TDS. Then the following conditions are equivalent:
\begin{enumerate}
\item $(X,T)$ is transitive and has dense small periodic sets.
\item $Trans_{b\F_{wt}}(X,T)=Trans(X,T)\neq\emptyset$.
\item $(X,T)$ is $b\F_{wt}$-point transitive.
\end{enumerate}
\end{thm}
\begin{proof}
(1)$\Rightarrow$(2) follows from Lemma \ref{lem:F-point-center} and Lemma \ref{lem:dsps}.

(2)$\Rightarrow$(3) is obvious.

(3)$\Rightarrow$(1) also follows from Lemma \ref{lem:dsps}.
\end{proof}

Denote $\F_{rs}=\{F\subset\Z:\, \exists k\in\N, \textrm{ s.t. } k\N\subset F\}$.
Let $(X,T)$ be a TDS, a point $x\in X$ is called a {\em quasi-periodic point}
if it is $\F_{rs}$-recurrent.

\begin{cor}
Let $(X,T)$ be an infinite minimal system. Then the following conditions are equivalent:
\begin{enumerate}
\item  $(X,T)$ has dense small periodic sets.
\item It is an almost one-to-one extension of some adding machine system.
\item It has some quasi-periodic point.
\item It is $b\F_{wt}$-point transitive.
\end{enumerate}
\end{cor}
\begin{proof}
(1)$\Leftrightarrow$(2)$\Leftrightarrow$(3) were proved in \cite{HLY09}
and (1)$\Leftrightarrow$(4) follows from Theorem \ref{thm:trans-periodic-set}.
\end{proof}

Recall that a system $(X,T )$ is called {\em totally transitive} if for every $k\in \N$, $(X,T^k)$ is transitive.
In \cite{HY05} Huang and Ye showed that a system which is totally transitive and
has dense small periodic sets is disjoint from every minimal system.
We call such a system {\em HY-system} for abbreviation.
It is not hard to see that an HY-system is a weakly mixing M-system (\cite{HY05}).

\begin{thm}\label{thm:HY-system}
Let $(X,T)$ be a TDS, then the following conditions are equivalent:
\begin{enumerate}
\item $(X,T)$ is an HY-system.
\item $Trans_{\mathcal F_{wt}}(X,T)=Trans(X,T)\neq\emptyset$.
\item $(X,T)$ is $\mathcal F_{wt}$-point transitive.
\item $(X,T^k)$ is an HY-system for every $k\in\N$.
\item $(X^k, T^{(k)})$ is an HY-system for every $k\in\N$,
where $X^k=X\times X\times \cdots\times X$($k$ times) and
$T^{(k)}=T\times T\times \cdots\times T$($k$ times).
\end{enumerate}
\end{thm}
\begin{proof}
(1)$\Rightarrow$(2). We need to show every transitive point is an $\F_{wt}$-transitive point.
Let $x$ be a transitive point and $U$ be an opene subset of $X$.
Since $(X,T )$ has dense small periodic sets,
there exists a closed subset $Y$ of $U$ and $k\in\mathbb N$
such that $T^k Y\subset Y$. Since $(X,T)$ is totally transitive,
$x$ is also a transitive point for the system $(X,T^k)$.
By the continuity of $T$, it is easy to see that $\{n\in\mathbb Z_+: (T^k)^nx\in U\}$ is thick.
Thus, $N(x,U)=\{n\in\Z: T^n x\in U\}$ is weakly thick, so $x$ is an $\F_{wt}$-transitive point.

(2)$\Rightarrow$(3) is obvious.

(3)$\Rightarrow$(4). We first prove the following Claim.

\smallskip
{\bf Claim}: If $F$ is weakly thick, then for every $k\in \mathbb N$,
$F_k=\{m\in\mathbb Z_+: mk\in F\}$ is also weakly thick.

{\bf Proof of the Claim}:
If $F$ is weakly thick, then there exists $r\in\mathbb N$ such that $F_r=\{m\in\mathbb Z_+: rm\in F\}$ is thick.
Then for every $k\in\N$,
\[\{m\in\mathbb Z_+: rm\in F_k\}=\{m\in\mathbb Z_+: krm\in F\}=\{m\in \Z: mk\in F_r\}\]
is also thick. Thus, $F_k$ is weakly thick.

To show that $(X,T^k)$ is an HY-system for every $k\in\N$,
it suffices to show that $(X,T^k)$ is transitive and has dense small periodic sets for every $k\in\N$.
Now fix $k\in\N$. Let $x\in X$ be an $\F_{wt}$-transitive point in $(X,T)$, then by the Claim,
$x$ also is an $\F_{wt}$-transitive point in $(X,T^k)$. Thus, $(X,T^k)$ is transitive
and by Lemma \ref{lem:dsps} $(X,T^k)$ has dense small periodic sets.

(4)$\Rightarrow$(1) and (5)$\Rightarrow$(1) are obvious.

(1)$\Rightarrow$(5). Since every HY-system is weakly mixing, $(X^k, T^{(k)})$ is totally transitive for every $k\in\N$.
Then it suffices to show that $(X^k, T^{(k)})$ has dense small periodic sets for every $k\in\N$.

Now fix $k\in\N$. Let $W$ be an opene subset of $X^k$, then there exist opene subsets $U_1,U_2,\ldots,U_k$ of $X$
such that $U_1\times U_2\times \cdots \times U_k\subset W$. Since $(X,T)$ has dense small periodic sets,
there exists closed subsets $Y_1,Y_2,\ldots, Y_k$ of $U_1, U_2,\ldots, U_k$ and $n_1, n_2,\ldots, n_k\in\N$
such that $T^{n_i}Y_i\subset Y_i$ for $i=1,2,\ldots,k$. Let $Y=Y_1\times Y_2\times\cdots\times Y_k$ and $n=n_1n_2\cdots n_k$.
Then $Y\subset W$ and $(T^{(k)})^{n}Y\subset Y$. Thus,  $(X^k, T^{(k)})$ has dense small periodic sets.
\end{proof}

\begin{rem}
(1). There exists a weakly mixing M-system which is not an HY-system (\cite{HY05}).

(2). There exists an HY-system without periodic points (\cite{HY05}).
\end{rem}

\section{Applications}
In this section as applications, we discuss disjointness and weak disjointness.
\subsection{Disjointness}
The notion of disjointness of two TDSs was introduced by Furstenberg his seminal
paper \cite{F67}. Let $(X, T)$ and $(Y, S)$ be two TDSs. We call $J\subset X\times Y$ is a {\em joining} of
$X$ and $Y$ if $J$ is a non-empty closed invariant set, and is projected onto $X$ and $Y$ respectively.
If each joining is equal to $X\times Y$ then we call $(X, T)$ and $(Y, S)$
are {\em disjoint}, denoted by $(X, T) \bot (Y, S)$ or $X\bot Y$. Note that if $(X, T)\bot (Y, S)$
then one of them is minimal, and if $(X, T)$ is minimal then the set of recurrent
points of $(Y, S)$ is dense (\cite{HY05}).

\begin{lem} \label{lem:disjoint}
Let $\F$ be a family.
If $(X,T)$ is $\F$-point transitive and $(Y,S)$ is minimal with $Rec_{\kappa\F}(Y,S)=Y$,
then $(X, T) \bot (Y, S)$.
\end{lem}
\begin{proof}
Let $J$ be a joining of $X$ and $Y$, and $x\in Trans_{\F}(X,T)$.
There exists some $y\in Y$ such that $(x,y)\in J$.
For every opene subset $U$ of $X$ and neighborhood $V$ of $y$,
we have $N(x, U)\in \F$ and $N(y, V)\in \kappa\F$,
then  $\overline{Orb((x,y), T\times S)}\cap U\times V\neq\emptyset$.
Therefore, $X\times \{y\}\subset  \overline{Orb((x,y), T\times S)}\subset J$.
By $J$ is $T\times S$-invariant and $Y$ is minimal, we have $J=X\times Y$.
Thus, $(X, T) \bot (Y, S)$.
\end{proof}

Recall that a system is called {\em distal} if there is no proper proximal pairs.

\begin{thm}\cite[Theorem 9.11]{F81}
A minimal system $(X,T)$ is distal if and only if $Rec_{\kappa\F_{ip}}(X,T)=X$
\end{thm}

\begin{cor}
Every $\F_{ip}$-point transitive system is disjoint from every minimal distal system.
\end{cor}

Recall that a system $(X,T)$ is called {\em totally minimal} if for every $k\in\N$, $(X,T^k)$ is minimal.

\begin{thm}
(1). Every HY-system is disjoint from every minimal system (\cite{HY05}).

(2). Every transitive system with dense small periodic sets
is disjoint from every totally minimal system.
\end{thm}
\begin{proof}
(1). By Theorem \ref{thm:HY-system}, every HY-system is $\F_{wt}$-point transitive.
Then by Lemma \ref{lem:disjoint} it suffices to show that for every minimal system $(X,T)$ we have
$Rec_{\kappa\F_{wt}}(X,T)=X$.

Let $(X,T)$ be a minimal system, $x\in X$ and $U$ be a neighborhood of $x$.
Then it suffices to show that for every $F\in\F_{wt}$, $N(x, U)\cap F\neq\emptyset$.
Let $F\in\F_{wt}$, there exists $k\in\N$
such that $\{n\in\N: kn\in F\}$ is thick. Since $x$ is a minimal point in $(X,T)$,
$x$ also is a minimal point $(X,T^k)$. Then $\{n\in\N: (T^k)^nx\in U\}$ is syndetic.
Therefore, $N(x,U)\cap F\neq\emptyset$.

(2). By Theorem \ref{thm:trans-periodic-set},
every transitive system with dense small periodic sets is $b\F_{wt}$-point transitive.
Then by Lemma \ref{lem:disjoint} it suffices to show that for every totally minimal system $(X,T)$
we have $Rec_{\kappa b\F_{wt}}(X,T)=X$. We first prove the following Claim.

\smallskip
{\bf Claim}: If $F\in b\F_{wt}$, there exists $q\in \Z$ such that $-q+F\in \F_{wt}$.

{\bf Proof of the Claim}:
Let $F\in b\F_{wt}$, then there exist two sequences $\{a_i\}_{i=1}^\infty$, $\{b_i\}_{i=1}^\infty$ in $\Z$
and $k\in\N$ such that $\bigcup_{i=1}^\infty(a_i + k[b_i, b_i+i])\subset F$.
Without lose of generality, assume that there exists $q\in [0,k)$ such that for every $i\in N$
$a_i\equiv q \pmod k$. Let $a_i=kc_i +q$ and $F'=\bigcup_{i=1}^\infty k[b_i+c_i, b_i+c_i+i]$,
then $F'$ is weakly thick and $q+F'\subset F$.
Thus, $-q+F$ is weakly thick.

Let $(X,T)$ be a totally minimal system, $x\in X$ and $U$ be a neighborhood of $x$.
Then it suffices to show that for every $F\in b\F_{wt}$, $N(x, U)\cap F\neq\emptyset$.
Let $F\in b\F_{wt}$, by the Claim there exists $q\in\Z$
such that $-q+F$ is weakly thick. Let $F'=-q+F$.
Since $(X,T)$ is totally minimal, $N(T^px,U)\cap F'\neq\emptyset$,
then $N(x,U)\cap F\neq\emptyset$ since $N(T^px,U)+p\subset N(x,U)$.
\end{proof}

\subsection{Weak disjointness}
Let $(X, T)$ and $(Y, S)$ be two TDSs, they are called {\em weakly disjoint} if $(X\times Y, T\times S)$ is
transitive, denoted by $(X, T) \curlywedge (Y, S)$ or $X\curlywedge Y$ (\cite{P72}).
It is easy to see that if $(X, T)$ and $(Y, S)$ are weakly disjoint, then both of them are transitive.

\begin{lem}\label{lem:weakly-disjoint}
Let $\F$ be a family.
If $(X,T)$ is $\F$-transitive and $(Y,S)$ is transitive and $\kappa\F$-center,
then $(X, T) \curlywedge (Y, S)$.
\end{lem}
\begin{proof}
Let $U_1, U_2$ be two opene subsets of $X$ and  $V_1, V_2$ be two opene subsets of $Y$.
Since $(Y,S)$ is transitive, then there exists $n\in\mathbb N$ such that $V_1\cap S^{-n}V_2\neq\emptyset$.
Let $V=V_1\cap S^{-n}V_2$. Then it is easy to see that
\[ n+N(U_1, T^{-n}U_2)\cap N(V,V)\subset N(U_1\times V_1, U_2\times U_2).\]
Since $(X,T)$ is $\F$-transitive and $(Y,S)$ is $\kappa\F$-center,
$N(U_1, T^{-n}U_2)\cap N(V,V)\neq\emptyset$.
Then $N(U_1\times V_1, U_2\times U_2)\neq\emptyset$.
Thus, $(X\times Y, T\times S)$ is transitive, \ie, $(X, T)\curlywedge (Y, S)$.
\end{proof}

\begin{lem}\label{lem:wt-P}
Let $F\subset \Z$ be weakly thick, then there exists a P-system $(X,T)$,
a transitive point $x\in X$ and an opene subset $U$ of $X$ such that
$N(x,U)\subset F$.
\end{lem}

\begin{proof}
Let $F\subset \Z$ be a weakly thick set, there exists $k\in\N$ and
a sequence $\{a_n\}_{n=1}^\infty$ of $\Z$ such that
$\bigcup_{n=1}^\infty(k[a_n, a_n+n])\subset F$.
Without lose of generality, assume that $a_{n+1}>a_n+2n$.
Let $F_0=\bigcup_{n=1}^\infty k[a_n,a_n+n]$.
We will construct $x^{(n)}=\mathbf{1}_{A_n}\in \{0,1\}^{\Z}$ such that
$A_n\subset F_0$ and $x=\lim x^{(n)}=\mathbf{1}_A$.
Moreover, $X=\overline{Orb(x,T)}$ is a P-system and
$N(x,U)=A\subset F_0$, where $U=\{y\in X:\,y(0)=1\}$.

To obtain $x^{(n)}$ we construct a finite word $B_n$ such that
$x^{(n)}$ begins with $B_n$ and $B_n$ appears in $x^{(n)}$ weakly thick,
and in the next step let $B_{n+1}$ begin with $B_n$.
Let $\{P_i\}_{i=0}^\infty$ is a partition of $\N$ and each $P_i$ is infinite.

\smallskip
{\bf Step 1:} Construct $x^{(1)}$.
Let $B_1=\mathbf{1}_{F_0}[0, ka_1+k-1]$ and $r_1=|B_1|$ be the length of $B_1$.
Put $x^{(1)}[0, r_1-1]=B_1$.
For every $p\in P_1$, if $kp\geq r_1$, let $l$ be the integer part of $\frac{kp}{r_1}$,
put $x^{(1)}[ka_p+jr_1, ka_p+(j+1)r_1-1]=B_1$ for $j=0,1,\ldots, l-1$ and
$x^{(1)}(i)=0$ for other undefined position $i$.
Let $A_1\subset \Z$ such that $\mathbf 1_{A_1}=x^{(1)}$, then $A_1\subset F_0$.

\smallskip
{\bf Step 2:} Construct $x^{(2)}$.
Let $B_2=x^{(1)}[0, ka_2-1]$ and $r_2=|B_2|$ be the length of $B_2$.
Put $x^{(2)}[0, r_2-1]=B_2$.
For every $p\in P_2$, if $kp\geq r_2$, let $l$ be the integer part of $\frac{kp}{r_2}$,
put $x^{(2)}[ka_p+jr_2, ka_p+(j+1)r_2-1]=B_2$ for $j=0,1,\ldots, l-1$ and
$x^{(2)}(i)=x^{(1)}(i)$ for other undefined position $i$.
Let $A_2\subset \Z$ such that $\mathbf 1_{A_2}=x^{(2)}$, then $A_1\subset A_2\subset F_0$.

\smallskip
{\bf Step 3:} If $x^{(n)}$ has been constructed, now construct $x^{(n+1)}$.
Let $B_{n+1}=x^{(n)}[0, ka_{n+1}-1]$ and $r_{n+1}=|B_{n+1}|$ be the length of $B_{n+1}$.
Put $x^{(n+1)}[0, r_{n+1}-1]=B_{n+1}$.
For every $p\in P_{n+1}$, if $kp\geq r_{n+1}$,
let $l$ be the integer part of $\frac{kp}{r_{n+1}}$,
put $x^{(n+1)}[ka_p+jr_{n+1}, ka_p+(j+1)r_{n+1}-1]=B_{n+1}$ for $j=0,1,\ldots, l-1$ and
$x^{(n+1)}(i)=x^{(n)}(i)$ for other undefined position $i$.
Let $A_{n+1}\subset \Z$ such that $\mathbf 1_{A_{n+1}}=x^{(n+1)}$,
then $A_n\subset A_{n+1}\subset F_0$.

In such a way, let $x=\lim x^{(n)}=\mathbf{1}_A$ and $X=\overline{Orb(x,T)}$,
then $x$ is a recurrent point and $N(x,U)=A\subset F_0$.
It is easy to see that $X$ has dense periodic points.
This completes the  proof.
\end{proof}

\begin{thm}
Let $(X,T)$ be a TDS, then the following conditions are equivalent:
\begin{enumerate}
\item $(X,T)$ is $\Delta^*(\F_{wt})$-transitive.
\item $(X,T)$ is weakly disjoint from every transitive and $\Delta(\F_{wt})$-center system.
\item $(X,T)$ is weakly disjoint from every $b\F_{wt}$-point transitive system.
\item $(X,T)$ is weakly disjoint from every P-system.
\end{enumerate}
\end{thm}
\begin{proof}
(1)$\Rightarrow$(2) follows from Lemma \ref{lem:weakly-disjoint}.

(2)$\Rightarrow$(3). Let $(Y,S)$ be a $b\F_{wt}$-point transitive system.
It suffices to show that $(Y,S)$ is $\Delta(\F_{wt})$-center.
Let $U$ be an opene subset of $Y$, there exists a $b\F_{wt}$-transitive point $y\in U$.
It is easy to see that $N(U,U)=N(y,U)-N(y,U)$.
Then $N(U,U)\in \Delta(\F_{wt})$ since $N(y,U)\in b\F_{wt}$.

(3)$\Rightarrow$(4) follows from the fact that
every P-system is $b\F_{wt}$-point transitive.

(4)$\Rightarrow$(1). Let $U,V$ be two opene subsets of $X$.
 For every $F\in \F_{wt}$, by Lemma \ref{lem:wt-P}
there exists a P-system $(Y,S)$, a transitive point $x\in X$ and an opene subset $W$ of $X$
such that $N(x,W)\subset F$. Then $N(W,W)\subset (F-F)\cup\{0\}$.
Since  $(X,T)$ is weakly disjoint from $(Y,S)$, $N(U,V)\cap N(W,W)$ is infinite,
then $ N(U,V)\cap (F-F)\neq\emptyset$.
Thus, $N(U,V)\in \Delta^*(\F_{wt})$.
\end{proof}

%%%%%%%%%%%%%%%%%%%%%%%%%%%%%%%%%%
\section*{Appendix A}
In this appendix, as announced before, we construct an $\F_{ip}$-point transitive system which is not weakly mixing.
\begin{proof}[Proof of Proposition \ref{prop:Fip-not-wm}]
Let $\Sigma=\{0,1\}^{\mathbb Z_+}$ and $T$ be the shift map.
We will construct a sequence $W_1\subset W_2\subset W_3\cdots\subset \Z$
and  $W=\bigcup_{n=1}^\infty W_n$.
Let $x=\mathbf 1_W$ and $X=\overline{Orb(x,T)}$,
we want to show that $(X,T)$ is $\F_{ip}$-point transitive but not weakly mixing.

Let $p_0=0$ and $p^{(i)}_{j,0}=0$ for all $i,j\in\N$. Let $W_0=\{p_0\}=A_0$. Choose $p_1>2$ and set
\[W_1=W_0\cup \{p_1\}.\]

Arrange $W_2$ as
\[W_1=A_0\cup A_1,\]
where $A_1=\{p_1\}$. Let $k_1=1$,
then $\min A_i>3\max A_j +2$ for $0\leq j<i\leq k_1$.

Choose $p_2$ and $p^{(1)}_{1,1}$ are positive integers to be defined later and set

\[W_2=W_1\bigcup\left(W_1+p_2\right)
\bigcup\left(W_1+p^{(1)}_{1,1}-1\right).\]

Arrange  $W_2$ as

\[W_2=W_1\bigcup A_2 \bigcup A_3.\]

Let $k_2=3$, choose appropriate $p_2$ and $p^{(1)}_{1,1}$ such that
$\min A_i>3\max A_j +2$ for $0\leq j<i\leq k_2$.

Choose $p_3$, $p^{(1)}_{1,2}$, $p^{(2)}_{1,1}$ and $p^{(2)}_{2,1}$
are positive integers to be defined later and set
\begin{eqnarray*}
W_3=W_2&\bigcup&\left(W_2+p_3\right)
\bigcup \left(W_1+p^{(1)}_{1,2}-1+FS\{p^{(1)}_{1,i}\}_{i=0}^1\right)\\
&\bigcup&\left(W_2+p^{(2)}_{1,1}-1\right)
\bigcup \left(W_2+p^{(2)}_{2,1}-2\right).
\end{eqnarray*}

Arrange  $W_3$ as

\[W_3=W_2\bigcup A_{4}\bigcup A_{5}\bigcup A_{6}\bigcup A_{7}.\]

Let $k_3=7$,
choose appropriate  $p_3$, $p^{(1)}_{1,2}$, $p^{(2)}_{1,1}$ and $p^{(2)}_{2,1}$
such that
$\min A_i>3\max A_j +2$ for $0\leq j<i\leq k_3$.

Assume that $W_n$ has been constructed, now construct $W_{n+1}$,
choose $p_{n+1}$, $p^{(1)}_{1,n}$, $p^{(2)}_{1,n-1}$, $p^{(2)}_{2,n-1}$,
$\ldots$, $p^{(n)}_{n,1}$
are positive integers to be defined later and  set

\begin{eqnarray*}
W_{n+1}=&W_n&\bigcup\left(W_n+p_{n+1}\right)\bigcup
\left(W_1+p^{(1)}_{1,n}-1+FS\{p^{(1)}_{1,i}\}_{i=0}^{n-1}\right)\\
&&\bigcup\left(\bigcup_{j=1}^2\left(W_2+p^{(2)}_{j,n-1}-j+FS\{p^{(2)}_{j,i}\}_{i=0}^{n-2}\right)\right)\bigcup\cdots\\
&&\bigcup\left(\bigcup_{j=1}^{n-1} \left(W_{n-1}+p^{(n-1)}_{j,2}-j+FS\{p^{(n-1)}_{j,i}\}_{i=0}^1\right)\right)\\
&&\bigcup\left(\bigcup_{j=1}^{n} \left(W_n+p^{(n)}_{j,1}-j\right)\right).
\end{eqnarray*}

Arrange $W_{n+1}$ as
\[ W_{n+1}=W_{n}\bigcup A_{k_{n}+1}\bigcup A_{k_{n}+2}\bigcup
\cdots\bigcup A_{k_{n+1}},\]
where $A_s$ is $W_n+p_{n+1}$ or in the form of $W_r+p^{(r)}_{j,t}-j+FS\{p^{(r)}_{j,i}\}_{i=0}^{t-1}$.

Choose appropriate  $p_{n+1}$, $p^{(1)}_{1,n}$, $p^{(2)}_{1,n-1}$,
$p^{(2)}_{2,n-1}$,
$\ldots$, $p^{(n)}_{n,1}$
such that
$\min A_i>3\max A_j +2$ for $0\leq j<i\leq k_{n+1}$.

Let $W=\bigcup_{n=0}^\infty W_n$. For every $n\in\N$, let
$P^{(n)}_0=FS\{p_j\}_{j=n+1}^\infty$,
$P^{(n)}_i=FS\{p^{(n)}_{i,j}\}_{j=1}^\infty$ for $i=1,\ldots,n$.
Then for every $n\in \N$, $W_n+P^{(n)}_i-i\subset W$ for $i=0,1,\ldots, n$.

Let $x=\mathbf 1_W$ and $X=\overline{Orb(x,T)}$.
First, we show that $(X,T)$ is $\F_{ip}$-point transitive.
It suffices to show that $x$ is an $\F_{ip}$-transitive point.
Let $r_n=\max W_n$ and
$[W_n]=\{y\in X: $ for $i\in [0, r_n]$, $y(i)=1$ if and only if $i\in W_n\}$.
Then $\{[W_n]: n\in \N\}$ is a neighborhood base of $x$. It is easy to see that
\[N(x, [W_n])\supset \bigcup_{i=0}^n (P^{(n)}_i-i).\]
For every opene subset $U$ of $X$, there exists $k\in \Z$ such that $T^kx\in U$.
By the continuity of $T$, there exists $n>k$ such that $T^k([W_n])\subset U$,
so
\[P^{(n)}_k  \subset k+N(x, [W_n])\subset N(x,U).\]
Thus, $x$ is an $\F_{ip}$-transitive point.

To see that $(X,T)$ is not weakly mixing, it suffices to show that $N([1], [1])=W-W$ is not thick.
In fact, we show that for every $n\geq 1$, if $a\neq b\in W_n-W_n$ then $|a-b|>2$.

{\bf Claim 1}: If $a\in A_i-A_i$ for some $i\in [k_n+1, k_{n+1}]$ is not zero, then
there exist $1\leq i_1< i_2 \leq k_n$ such that $a\in A_{i_2}-A_{i_1}$.

We prove this claim by induction of $n$.  It is easy to see that
the result holds for $n=1, 2$.  Now assume that the result holds for  $1,2,\ldots,n-1$.
Let $a\in A_i-A_i$ for some $i\in [k_n+1, k_{n+1}]$.
By the construction, every element in $A_i$ has the same part in
$\{$ $p_{n+1}$, $p^{(1)}_{1,n}-1$, $p^{(2)}_{1,n-1}-1$,
$p^{(2)}_{2,n-1}-2$, $\ldots$, $p^{(n)}_{n,1}-n$$\}$, then
$a\in W_{n-1}-W_{n-1}$, by the assume the proof is complete.

{\bf Claim 2}: For every $n\geq 1$, if $a\neq b\in W_n-W_n$ then $|a-b|>2$.

Again we prove this claim by induction of $n$. It is easy to see that
the result holds for $n=1, 2$.
Now assume that the result holds for  $1,2,\ldots,n-1$.

Let $a\neq b\in W_n-W_n$. There exist $s_1, s_2, t_1$ and $t_2\in [0, k_n]$
such that $a\in A_{s_1}-A_{s_2}$ and $b\in A_{t_1}-A_{t_2}$.
By the Claim 1, we can assume that $s_1>s_2$ and $t_1> t_2$.

\begin{enumerate}
\item Case $s_1>t_1$.  Then $a-b\geq \min A_{s_1}-3\max A_{s_1-1}>2$.
\item Case $s_1=t_1$.
 \begin{enumerate}
 	\item $s_2=t_2$. Then $a-b\in (A_{s_1}-A_{s_2})- (A_{t_1}-A_{t_2})
		=(A_{s_1}-A_{t_1})-(A_{s_2}-A_{t_2})$.
		By the Claim 1, we have  $a-b=c-d$ for some $c,d\in W_{n-1}-W_{n-1}$,
		so $a-b>2$ by induction.
	\item $s_2>t_2$. Then $a-b\in (A_{s_1}-A_{s_2})- (A_{t_1}-A_{t_2})
		=(A_{s_1}-A_{t_1})-(A_{s_2}-A_{t_2})$. By the Claim 1, we have
        $A_{s_1}-A_{t_1}\subset W_{n-1}-W_{n-1}$.
		If $s_2> k_{n-1}$, this turn to the case (1), so $|a-b|>2$.
		If $s_2\leq k_{n-1}$, then $|a-b|>2$ by induction.
 \end{enumerate}
\end{enumerate}
Hence, $(X,T)$ is as required.
\end{proof}

%%%%%%%%%%%%%%%%%%%%%%%%%%%%%%%%%%%%%%%%%%%%%%%%%%%%%


\begin{thebibliography}{99}
\bibitem{A97} E. Akin, \textit{Recurrence in Topological Dynamical Systems.
Families and Ellis Actions}, The University Series in Mathematics. Plenum Press, New York, 1997.

\bibitem{A04} E. Akin, \textit{Lectures on Cantor and Mycielski sets for dynamical systems},
Chapel Hill Ergodic Theory Workshops, 21--79, Contemp. Math., 356, Amer. Math. Soc., Providence, RI., 2004.

\bibitem{AG01} E. Akin and E. Glasner, \textit{Residual properties and almost equicontinuity},
J. Anal. Math. 84 (2001),  243--86.

\bibitem{AK03} E. Akin and S. Kolyada, \textit{Li-Yorke sensitivity}, Nonlinearity, 16 (2003), 1421--1433.

\bibitem{BD08} V. Bergelson and T. Downarowicz,
\textit{Large sets of integers and hierarchy of mixing properties of measure-preserving systems},
Colloquium Mathematicum 110 (2008), no. 1, 117-150.

\bibitem{BHM00} F. Blanchard, B. Host and A. Maass, \textit{Topological complexity},
Ergod. Theory Dyn. Syst. 20 (2000), 641--62.

\bibitem{F67} H. Furstenberg, \textit{Disjointness in ergodic theory, minimal sets,
and a problem in Diophantine approximation}, Math. Systems Theory, 1 (1967), 1--49.

\bibitem{F81} H. Furstenberg, \textit{Recurrence in ergodic theory and combinatorial number theory},
M. B. Porter Lectures. Princeton University Press, Princeton, N.J., 1981.
\bibitem{G04} E. Glasner, \textit{Classifying dynamical systems by their recurrence properties},
Topol. Methods Nonlinear Anal. 24 (2004), 21--40.

\bibitem{GH55} W. H. Gottschalk and G. A. Hedlund, \textit{Topological Dynamics},
Amer. Math. Soc. Collooquium Publications, Vol.36. Providence, R.I., 1955.

\bibitem{HZ02} W. He and Z. Zhou, \textit{A topologically mixing system with its measure center being a singleton},
Acta Math. Sinica, 45(5)(2002), 929-934.

\bibitem{HLY09}W. Huang, H. Li and X. Ye, \textit{Family-independence for topological and measurable dynamics},
to appear in Trans. Amer. Math. Soc, \url{http://arxiv.org/abs/0908.0574}.

\bibitem{HPY07} W. Huang, K. K. Park amd X. Ye, \textit{Dynamical systems disjoint from
all minimal systems with zero entropy}, Bull. Soc. Math. France, 135 (2007), 259--282.

\bibitem{HY02}W. Huang and X. Ye,
\textit{An explicit scattering, non-weakly mixing example and weak disjointness},
Nonlinearity 15, (2002), 849--862.

\bibitem{HY04} W. Huang and X. Ye,
\textit{Topological complexity, return times and weak disjointness},
Ergod. Theor. Dyn. Syst., 24 (2004), 825--846.

\bibitem{HY05}W. Huang and X. Ye, \textit{Dynamical systems disjoint from any minimal system},
Trans. Amer. Math. Soc., 357, (2005), no. 2, 669--694.

\bibitem{P72} R. Peleg, \textit{Weak disjointness of transformation groups},
Proc. Amer. Math. Soc., 33 (1972), 165--170.

\bibitem{SY04} S. Shao and X. Ye, \textit{$\F$-mixing and weakly disjointness}.
Topology and its applications, 135 (2004), no.1--3 , 231--247.

\bibitem{W00} B. Weiss,  \textit{Single Orbit Dynamics},
Amer. Math. Soc., Regional Conference Series in Mathematics, No. 95, Providence, RI, 2000.
\end{thebibliography}
\end{document}